\documentclass[11pt]{amsart}
\usepackage{amscd,amssymb,verbatim}

\theoremstyle{plain}
\newtheorem{thm}{Theorem}
\newtheorem{cor}[thm]{Corollary}
\newtheorem{lem}[thm]{Lemma}
\newtheorem{prop}[thm]{Proposition}

\newcommand{\ep}{\varepsilon}

\newcommand{\om}{\omega}

\newcommand{\Sig}{\Sigma}
\newcommand{\vp}{\varphi}

\newcommand{\bs}{\backslash}

\newcommand{\ol}{\overline}

\newcommand{\inte}{\operatorname{int}}
\newcommand{\Lip}{\operatorname{Lip}}
\newcommand{\lip}{\operatorname{lip}}

\newcommand{\N}{{\mathbb N}}
\newcommand{\R}{{\mathbb R}}

\newcommand{\D}{{\mathcal D}}

\newcommand{\cD}{{\mathcal D}}

\newcommand{\C}{{\mathbb C}}
\newcommand{\U}{{\mathcal U}}

\begin{document}

\title{Biseparating maps on generalized Lipschitz spaces}

\begin{abstract}
Let $X, Y$ be complete metric spaces and $E, F$ be Banach spaces.
A bijective linear operator from a space of $E$-valued functions on $X$ to a space of $F$-valued functions on $Y$ is said to be biseparating if $f$ and  $g$ are disjoint if and only if $Tf$ and $Tg$ are disjoint. We
introduce the class of generalized Lipschitz spaces, which includes
as special cases the classes of Lipschitz, little Lipschitz and
uniformly continuous functions. Linear biseparating maps between
generalized Lipschitz spaces are characterized as weighted
composition operators, i.e., of the form $Tf(y) = S_y(f(h^{-1}(y))$
for a family of vector space isomorphisms $S_y: E \to F$ and a
homeomorphism $h : X\to Y$. We also investigate the continuity of
$T$ and related questions.  Here the functions involved (as well as
the metric spaces $X$ and $Y$) may be unbounded.  Also, the
arguments do not require the use of compactification of the spaces
$X$ and $Y$.
\end{abstract}

\author{Denny H. Leung}
\address{Department of Mathematics, National University of Singapore, 2 Science Drive
2, Singapore 117543.}
\email{matlhh@nus.edu.sg}
\thanks{Research partially supported by AcRF project no.\ R-146-000-086-112}

\maketitle

%%%%%%%%%%%%%%%%%%%%%%%%%%%%%%%%%%%%%%%%%%%%%%%%%%%%%
\section{Introduction}

In his classical treatise, {\em Th\'{e}orie des Op\'{e}rations
Lin\'{e}aires} \cite{B}, Banach proved that the linear isometric
structure of the Banach space $C(X)$ of continuous functions  on a
compact metric space determines the space $X$ up to homeomorphism.
The result was generalized by Stone \cite{St} to general compact
Hausdorff spaces $X$.  Subsequently, Gelfand and Kolmogorov
\cite{GK} and Kaplansky \cite{K} showed that $X$ is also determined
up to homeomorphism by the algebraic structure and the lattice
structure of $C(X)$ respectively. In the intervening decades, these
types of results have been generalized to many other classes of
function spaces and also to spaces of vector-valued functions.  The
classic monograph \cite{GJ} considers the relationship between the
algebraic structure of spaces of continuous functions on $X$ and the
space $X$ itself for general classes of topological spaces.  The
work \cite{Be} uses the Banach-Stone Theory of vector-valued
continuous functions $C(X,E)$ as a tool to study the Banach space
$E$ itself, leading to the theory of $M$- and $L$-structures of
Banach spaces. For the general theory of isometries on Banach
spaces, we refer the reader to the two-volume monograph of Fleming
and Jamison \cite{FJ}.  For a survey on various aspects of research
surrounding Banach-Stone type theorems, see \cite{GJ2}.

A useful unifying notion that has been introduced into the theory is
that of separating or biseparating maps.  Two functions $f$ and $g$
defined on the same domain $X$ with values in a vector space are
said to be {\em disjoint} if for all $x \in X$, either $f(x) = 0$ or
$g(x)= 0$.  A map $T$ between vector-valued function spaces is {\em
separating} (also called {\em disjointness preserving} or a {\em
Lamperti operator}) if $T$ maps disjoint functions to disjoint
functions.  It is {\em biseparating} if $T$ is invertible and both
$T$ and $T^{-1}$ are separating. Clearly, algebraic or lattice
homomorphisms (isomorphisms) are separating (biseparating). In many
instances, isometries between Banach function spaces can also be
shown to be biseparating.  This explains the interest and amount of
work devoted to the characterization of separating or biseparating
operators.  See, e.g., \cite{A1,A2,A2-1,A3,ABN,AK1,AK2,GJW,JL,LW}.

The study of Lipschitz spaces can be traced back to de Leeuw
\cite{dL} and Sherbert \cite{S1,S2} for the scalar case, and Johnson
\cite{J} for the vector-valued case.  A survey on the algebra of
Lipschitz functions can be found in \cite{W}.  Recent work on
separating and biseparating maps on Lipschitz spaces and spaces of
uniformly continuous functions include \cite{AD,GJ4, GJ3, J-V, J-V2,
JVW,JW}.  In particular, characterizations of biseparating maps on
spaces of {\em bounded} Lipschitz or little Lipschitz functions are
obtained in \cite{AD,JVW,JW}. In this paper, we consider spaces of
functions determined by the ``modulus of continuity" and call such
classes {\em generalized Lipschitz spaces}.  This notion serves to
unify the study of spaces of Lipschitz, little Lipschitz and
uniformly continuous functions.  One of the main aims of this paper is to characterize all
biseparating operators between generalized Lipschitz spaces. We  make use of a new approach
that bypasses the usual compactification procedures, and is rather
more closely tied to the metric structure of the underlying spaces.
(See \S 2.) The second critical ingredient in our argument is the
construction of ``bump" functions (Lemma \ref{bump}). Taking
advantage of such ``bumps" allows us to complete the
characterization of biseparating maps as weighted composition
operators (Theorem \ref{thm11}). In \S 3, we consider questions
connected with automatic continuity.

Let $(X,d)$ be a complete metric space and $E$ be a real or complex Banach space.
For a function $f:X\to E$, its {\em modulus of continuity} is the
function $\om_f:[0,\infty)\to [0,\infty]$ defined by
\[ \om_f(\ep) = \sup\{\|f(x_1)-f(x_2)\|: d(x_1,x_2) \leq \ep\}.\]
Note that  $f$ is uniformly continuous on $X$ if and only if $\om_f$
is continuous at $0$. In general, we say that $\sigma:
[0,\infty)\to[0,\infty]$ is a {\em modulus function} if $\sigma$ is
nondecreasing, $\sigma(0) = 0$ and $\sigma$ is continuous at $0$.  A
nonempty set $\Sig$ of modulus functions is called a {\em modulus
set} if
\begin{enumerate}
\item[(MS1)] If $\sigma_1,\sigma_2$ belong to $\Sig$,
then there exist $\sigma \in \Sig$ and $K < \infty$ such that
$\sigma_1+\sigma_2 \leq K \sigma$,
\item[(MS2)] For every sequence $(\sigma_n)$ in $\Sig$ and every
nonnegative summable real sequence $(a_n)$, there are a $\sigma \in \Sig$ and $K < \infty$
so that $\sum a_n(\sigma_n\wedge 1) \leq
K\sigma$.
\end{enumerate}
Let $\Sig$ be a modulus set.  Define the {\em generalized Lipschitz
space} $\Lip_\Sig(X,E)$ to the the set of all functions $f : X \to
E$ such that $\om_f \leq K \sigma$ for some $\sigma \in \Sig$ and $K
< \infty$.  Since $\om_{cf_1+f_2} \leq |c|\om_{f_1} + \om_{f_2}$, it
follows from (MS1) that $\Lip_\Sig(X,E)$ is a vector space.
We reiterate that all functions in $\Lip_\Sig(X,E)$ are necessarily uniformly continuous.  Also $\Lip_\Sig(X,E)$ always contains all constant functions.
When $E
= \R$ or $\C$, $\Lip_\Sig(X,E)$ is abbreviated to $\Lip_\Sig(X)$. To justify
the introduction of this new class of spaces, let us look at a few
examples.\\

\noindent {\bf Examples}.
\begin{enumerate}
\item If $\Sigma$ consists of the identity function $\sigma(t) = t$ only, then $\Lip_\Sigma(X,E)$ is the class of Lipschitz functions $\Lip(X,E)$.  Observe that if $0< \alpha < 1$ and we let $X^\alpha$ be the space $X$ with the metric $d^\alpha$, then $\Lip(X^\alpha,E)$ is the class $\Lip_\alpha(X,E)$ of Lipschitz functions (on $(X,d)$) of order $\alpha$.
\item If $\Sigma$ consists of all modulus functions $\sigma$ such that $\sigma(t) \leq t$ for all $t \geq 0$ and $\lim_{t\to 0}\sigma(t)/t = 0$, then $\Lip_\Sigma(X,E)$ is the small Lipschitz class $\lip(X,E)$.  Again, for $0 < \alpha < 1$, $\lip_\alpha(X,E) = \lip(X^\alpha,E)$.
\item If $\Sigma$ is the set of all modulus functions, then $\Lip_\Sig(X,E)$ is the space of uniformly continuous functions $\U(X,E)$ from $X$ to $E$.
\item If $\Sig$ is a modulus set and $\Sig_b = \{\sigma\wedge 1: \sigma \in \Sig\}$, then $\Lip_{\Sig_b}(X,E)$ is the set of all {\em bounded} functions in $\Lip_\Sig(X,E)$.
\end{enumerate}

A generalized Lipschitz space $\Lip_\Sig(X)$ is said to be {\em
Lipschitz normal} if for every pair of subsets $U, V$ of $X$ with
$d(U,V) >0$, there exists $f \in \Lip_\Sig(X)$, $0 \leq f \leq 1$,
such that $f=0$ on $U$ and $f=1$ on $V$.  We will say that
$\Lip_\Sig(X,E)$ is Lipschitz normal if $\Lip_\Sig(X)$ is. For any
metric space $X$, $\Lip(X)$, $\lip(X^\alpha)$, $0 < \alpha <1$, and $\U(X)$ are Lipschitz normal.
Another example is the following: $\lip(\Delta)$ is Lipschitz
normal, where $\Delta$ is the Cantor set with the usual metric. {\em In this paper, all generalized Lipschitz spaces considered are
assumed to be Lipschitz normal}.

If $f$ belongs to $\Lip_\Sig(X,E)$, let $C(f)$ be the set $\{x\in X:
f(x) \neq 0\}$ and denote its closure by $\overline{C}(f)$. If
$(Y,d')$ is a complete metric space, $F$ is a Banach space and
$\Sig'$ is a modulus set, we may define the space
$\Lip_{\Sigma'}(Y,F)$ as above.  A linear map $T: \Lip_\Sig(X,E) \to
\Lip_{\Sig'}(Y,F)$ is said to be {\em biseparating} if $T$ is a
bijection and for all $f, g \in \Lip_\Sig(X,E)$,
\[ C(f)\cap C(g) = \emptyset \text{ if and only if } C(Tf) \cap C(Tg) = \emptyset.\]

The author thanks Wee-Kee Tang for many stimulating conversations regarding the materials contained herein.

\section{The Boolean algebra of closures of open sets}

Let $X$ be a complete metric space with metric $d$.  Denote by $\cD(X)$ the
collection of all subsets $A$ of $X$ such that $A= \overline{\inte
A}$.  Equivalently, $A \in \cD(X)$ if and only if $A$ is the closure
of an open subset of $X$.  In particular, $\overline{C}(f) \in \cD(X)$ for every $f \in \Lip_\Sig(X,E)$. $\cD(X)$ is a Boolean algebra under
the order of set inclusion, with lattice operations
\[ A\vee B = A\cup B \quad \text{and} \quad A\wedge B = \overline{\inte A \cap \inte B}\]
for all $A, B \in \D(X)$.  The $0$ and $1$ elements of $\cD(X)$ are
$\emptyset$ and $X$ respectively; the complement of $A \in \cD(X)$
is $\neg A = \overline{A^c}$, where $A^c$ is the set theoretic
complement of $A$.  For basic properties of Boolean algebras we refer the reader to \cite{MB}.
We
begin with a simple but fundamental observation.

\begin{prop}\label{prop0}
Let $\vp$ be a function in $\Lip_\Sig(X)$ with values in $[0,1]$ and let $f \in \Lip_\Sig(X,E)$ be such that $\|f(x)\| \leq M$ for all $x \in C(\vp)$.  Then $\vp f \in \Lip_\Sig(X,E)$ and $\om_{\vp f} \leq \om_f +M\om_\vp$.
\end{prop}

\begin{proof}
Suppose that $d(x_1,x_2) \leq \ep$. If neither $x_1$ nor $x_2$ lies in $C(\vp)$, then $\|(\vp f)(x_1) - (\vp f)(x_2)\| = 0$.  Otherwise, we may assume that $x_2 \in C(\vp)$ and hence $\|f(x_2)\| \leq M$.  Therefore,
\begin{align*}
\|(\vp f)(x_1) - &(\vp f)(x_2)\| \\&\leq |\vp(x_1)| \|f(x_1)-f(x_2)\| + |\vp(x_1)-\vp(x_2)| \|f(x_2)\|\\&
\leq \om_f(\ep) + M\om_{\vp}(\ep).
\end{align*}
\end{proof}

The next lemma is similar to Lemma 4.2 in \cite{A1}.

\begin{lem}\label{lem1.1}
Let $T: \Lip_\Sig(X,E) \to \Lip_{\Sig'}(Y,F)$ be a biseparating map. If $f,g \in
\Lip_\Sig(X,E)$ and $\overline{C}(f) \subseteq \overline{C}(g)$, then $C(Tf)
\subseteq \overline{C}(Tg)$.
\end{lem}

\begin{proof}
Suppose that $y \notin \overline{C}(Tg)$.  There exists $\ep > 0$ so
that $B(y,\ep) \cap \overline{C}(Tg) = \emptyset$ and that $Tf$ is
bounded on $B(y,\ep)$.  Let $\psi$ be a $[0,1]$-valued function in
$\Lip_{\Sig'}(Y)$ so that $\psi = 1$ on $B(y,\ep/2)$ and $\psi = 0$ outside
$B(y,\ep)$. Then $\psi Tf\in \Lip_{\Sig'}(Y,F)$ by Proposition \ref{prop0} and $C(\psi Tf)\cap C(Tg) =
\emptyset$.  Thus $C(T^{-1}(\psi Tf)) \cap C(g) = \emptyset$. Since
$C(T^{-1}(\psi Tf))$ is an open set, it follows that $ C(T^{-1}(\psi
Tf)) \cap \overline{C}(g) = \emptyset$, and hence $C(T^{-1}(\psi
Tf)) \cap C(f) = \emptyset$.  Therefore, $C(\psi Tf) \cap C(Tf) =
\emptyset$.  In particular, since $\psi(y) \neq 0$, we must have
$Tf(y) = 0$.  So $y \notin C(Tf)$.
\end{proof}

\begin{lem}\label{lem1.2}
For each open subset $U$ of $X$, there exists $f \in \Lip_\Sig(X,E)$ such that $C(f) = U$.
\end{lem}

\begin{proof}
For each $n\in \N$, let $U_n$ be the set of all $x \in X$ such that $d(x,U^c) \geq 1/n$.
Since $\Lip_{\Sig}(X,E)$ is Lipschitz normal, there exists $\vp_n \in \Lip_\Sig(X)$ with values in $[0,1]$ so that $\vp_n = 0$ on $U^c$ and $\vp_n = 1$ on $U_n$.  Take $\sigma_n \in \Sig$ and $K_n < \infty$ so that $\om_{\vp_n} \leq K_n\sigma_n$.  Note that $\om_{\vp_n} \leq 1$ as well.  So, by redefining the constant $K_n$ if necessary, we may assume that $\om_{\vp_n} \leq K_n(\sigma_n\wedge 1)$. The function $\vp = \sum \vp_n/(n^2(K_n+1))$ converges on $X$. Also,
\[ \om_\vp \leq \sum\frac{\om_{\vp_n}}{n^2(K_n+1)} \leq \sum\frac{\sigma_n\wedge 1}{n^2} \leq K\sigma\]
for some $\sigma \in \Sig$ and $K < \infty$ by condition (MS2) in the definition of modulus sets. Thus $\vp \in \Lip_\Sig(X)$.  Clearly $C(\vp) = U$.  Finally, choose any nonzero $u \in E$ and $f(x) = \vp(x)u$ is a function with the desired properties.
\end{proof}

\begin{prop}\label{prop1.2}
Let $T: \Lip_\Sig(X,E) \to \Lip_{\Sig'}(Y,F)$ be a biseparating map. For each $A \in \cD(X)$, let $\theta(A) = \overline{C}(Tf)$ for some
$f\in \Lip_\Sig(X,E)$ such that $C(f) = \inte A$.  Then $\theta$ is a
well-defined Boolean isomorphism from $\cD(X)$ onto $\cD(Y)$.
Moreover, for any $f\in \Lip_\Sig(X,E)$ and any $A \in \cD(X)$,  $f = 0$ on
$A$ if and only if $Tf = 0$ on $\theta(A)$.
\end{prop}

\begin{proof}
The fact that $\theta$ is well defined follows from Lemmas
\ref{lem1.1} and \ref{lem1.2}. By Lemma \ref{lem1.1}, $\theta$ preserves order.
Analogously, we can define $\tau: \cD(Y)\to \cD(X)$ by $\tau(B) =
\overline{C}(T^{-1}g)$ for some $g \in \Lip_{\Sig'}(Y,F)$ such that $C(g) =
\inte B$.  If $A \in \cD(X)$ and $f \in \Lip_\Sig(X,E)$ with $C(f) = \inte
A$, then $\theta(A) = \overline{C}(Tf)$.  Let $g \in \Lip_{\Sig'}(Y,F)$ be such
that $C(g) = \inte \theta(A)$. By Lemma \ref{lem1.1} applied to $T^{-1}$,
$\overline{C}(f) = \overline{C}(T^{-1}g)$.  Thus $A =
\overline{C}(f) = \tau(\theta(A))$.  Similarly, $\theta(\tau(B)) =
B$ for all $B \in \cD(Y)$.  Hence $\tau = \theta^{-1}$. Since both
$\theta$ and $\theta^{-1}$ are order preserving, $\theta$ is a
Boolean isomorphism.

If $f \in \Lip_\Sig(X,E)$ and $f = 0$ on $A \in \cD(X)$, then $C(f) \cap
C(f') = \emptyset$ for any $f' \in \Lip_\Sig(X,E)$ with $C(f') = \inte A$.
Hence $C(Tf)\cap C(Tf') = \emptyset$.  By continuity of $Tf$, $Tf =
0$ on $\overline{C}(Tf') = \theta(A)$. The converse follows by
symmetry.
\end{proof}

\section{Characterization of biseparating maps}

Let $X$ and $Y$ be complete metric spaces, $E$ and $F$ be Banach spaces, and $\Sig$ and $\Sig'$ be two modulus sets.
The closed unit ball of $F$ is denoted by $B_F$.  We begin with an easy observation.

\begin{lem}\label{lem1}
For any $a > 0$, the retraction $r: F \to aB_F$ defined by
\[ r(v) = \begin{cases}
          v &\text{ if $v \in aB_F$}\\
          \frac{av}{\|v\|} &\text{ otherwise}.
          \end{cases} \]
is a Lipschitz map with  $\om_r(t) \leq 2t$.
\end{lem}

\begin{lem}\label{bump}
Let $g$ be a function in $\Lip_{\Sig'}(Y,F)$. For all $a > 0$ and all $b \geq
2a$, there is a function $\tilde{g} \in \Lip_{\Sig'}(Y,F)$ with
$\om_{\tilde{g}} \leq 3\om_g$ such that
\[ \tilde{g}(y) = \begin{cases}
                  g(y) &\text{ if $\|g(y)\|\leq a$}\\
                  0 &\text{ if $\|g(y)\| \geq b$.}
                  \end{cases}
\]
\end{lem}

\begin{proof}
Let $r: F \to aB_F$ be the retraction defined in Lemma \ref{lem1}.
Then $\om_{r\circ g} \leq 2\om_g$ and $r\circ g$ is bounded in norm
by $a$.  For any $b \geq 2a$, the function $\gamma: [0, \infty) \to
[0,1]$ defined by
\[ \gamma(t) = \begin{cases}
               1 &\text{ if $t \in [0,a]$}\\
               \frac{b-t}{b-a} & \text{ if $t \in (a,b)$}\\
               0 &\text{ if $t \in [b,\infty)$}
               \end{cases}
\]
satisfies $\om_\gamma(t) \leq t/a$.  Let $\tilde{g}(y)
= \gamma(\|g(y)\|)r(g(y))$. Clearly, $\tilde{g}(y) = g(y)$ if
$\|g(y)\|\leq a$ and $0$ if $\|g(y)\| \geq b$.  For all $y_1, y_2
\in Y$ with $d'(y_1,y_2) \leq \ep$,
\begin{align*}
\|\tilde{g}(y_1) - \tilde{g}(y_2)\| & \leq \gamma(\|g(y_1)\|)\|r(g(y_1)) - r(g(y_2))\| \\
&\quad+ |\gamma(\|g(y_1)\|) - \gamma(\|g(y_2)\|)|\, \|r(g(y_2))\|\\
&\leq \om_{r\circ g}(\ep) + \frac{1}{a}\bigl|\|g(y_1)\|-\|g(y_2)\|\bigr|a \\
&\leq 3\om_g(\ep).
\end{align*}
Thus $\om_{\tilde{g}}\leq 3\om_g$ and $\tilde{g} \in \Lip_{\Sig'}(Y,F)$.
\end{proof}

\begin{lem}\label{lem2.1}
Let $(f_n)$ be a pairwise disjoint sequence of functions from $X$ into $E$. Assume that there is a modulus function $\sigma$ such that $\om_{f_n}\leq \sigma$ for all $n$.    Then the pointwise sum $f = \sum f_n$ satisfies $\om_f \leq 2\sigma$.
\end{lem}

\begin{proof}
For any $x_1, x_2 \in X$,
either there exists $n_1$ such that $f(x_i) = f_{n_1}(x_i)$, $i =
1,2$, or there are $n_1$ and $n_2$ so that $f(x_i) =
(f_{n_1}+f_{n_2})(x_i)$, $i = 1,2$.  It follows that $\om_f
\leq 2\sup_n\om_{f_n} \leq 2\sigma$.
\end{proof}

For the rest of the section, we consider a linear biseparating map $T: \Lip_{\Sig}(X,E) \to \Lip_{\Sig'}(Y,F)$.  Let
$\theta$ be the associated Boolean isomorphism from Proposition
\ref{prop1.2}.  If $u$ is a vector in $E$ or $F$, denote by $\hat{u}$ the constant function (defined on $X$ or $Y$) with value $u$. The next proposition is a key to subsequent arguments.

\begin{prop}\label{prop3}
For any $x_0 \in X$, there exists $f \in \Lip_{\Sig}(X,E)$ such that $f(x_0) \neq 0$ and $Tf$ is bounded on $Y$.
\end{prop}

\begin{proof}
Suppose that the proposition fails.  We have $x_0 \in X$ so that
$Tf$ is unbounded whenever $f(x_0) \neq 0$.  Pick any $u \in E \bs
\{0\}$ and let  $g
= T\hat{u}$.  First we need two lemmas.

\begin{lem}\label{lem2}
For all $a>0$ and all $\ep > 0$, $\theta^{-1}(\overline{\{\|g\|> a\}}) \wedge \overline{B(x_0,\ep)} \neq 0$.
\end{lem}

\begin{proof}
Suppose that $\theta^{-1}(\overline{\{\|g\|> a\}}) \wedge
\overline{B(x_0,\ep)} = 0$ for some $a, \ep > 0$.  Let $V =
\neg(\overline{\{\|g\|> a\}})$.  Then
\begin{align*}
& \inte V \cap \inte \overline{\{\|g\|> a\}} = \emptyset\\
\implies &\inte V \subseteq \{\|g\| \leq a\}\\
\implies &V = \overline{\inte V} \subseteq \{\|g\| \leq a\}.
\end{align*}
Since $g$ is uniformly continuous, $d'(V, {\{\|g\|\geq 2a\}}) > 0$.
Hence there exists $\psi\in\Lip_{\Sig'}(Y)$ such that $0 \leq \psi \leq 1$,
$\psi = 1$ on $V$ and $\psi = 0$ on ${\{\|g\|\geq 2a\}}$.  Note that
$g$ is bounded on $C(\psi)$ and hence $\psi g \in \Lip_{\Sig'}(Y,F)$ by Proposition \ref{prop0}.  As $\psi
g = g$ on $V$, we  have $T^{-1}(\psi g) = T^{-1}g = \hat{u}$ on
$\theta^{-1}(V)$ by Proposition \ref{prop1.2}.  From
$\theta^{-1}(\overline{\{\|g\|> a\}}) \wedge \overline{B(x_0,\ep)} =
0$ and keeping in mind the definition of $V$, we see that
$\overline{B(x_0,\ep)} \subseteq \theta^{-1}(V)$. In particular
$T^{-1}(\psi g)(x_0) = u \neq 0$.  By assumption, $\psi g$ is
unbounded on $Y$.  But this is clearly false since $\|(\psi g)(y)\|
\leq 2a$ for all $y \in Y$.
\end{proof}

\begin{lem}\label{lem3}
For all $a>0$ and all $\ep > 0$, there exists $b > a$ so that
$\theta^{-1}(\overline{\{\|g\|\in (a,b)\}}) \wedge
\overline{B(x_0,\ep)} \neq 0$.
\end{lem}

\begin{proof}
For each $n\in \N$, let $V_n = \overline{\{\|g\|\in (a,a+n)\}}$.
The sequence $(V_n)$ increases to $\overline{\{\|g\|> a\}}$ in
$\D(Y)$.  Hence $(\theta^{-1}(V_n))$ increases to
$\theta^{-1}(\overline{\{\|g\|> a\}})$ in $\D(X)$.  If the lemma
fails, $\theta^{-1}(V_n) \subseteq \neg(\overline{B(x_0,\ep)})$ for
all $n$ and thus $\theta^{-1}(\overline{\{\|g\|> a\}}) \subseteq
\neg(\overline{B(x_0,\ep)})$, contrary to Lemma \ref{lem2}.
\end{proof}

We now continue with the proof of Proposition \ref{prop3}.   Let
$(\ep_n)$ be a positive null sequence and set  $a_1=1$.  By Lemma
\ref{bump}, there exist $b_1 > a_1$ and $\tilde{g}_1 \in \Lip_{\Sig'}(Y,F)$
with $\om_{\tilde{g}_1} \leq 3\om_g$ such that
\[ \tilde{g}_1(y) = \begin{cases}
                  g(y) &\text{ if $\|g(y)\|\leq a_1$}\\
                  0 &\text{ if $\|g(y)\| \geq b_1$.}
                  \end{cases}
\]
In general, after $a_n, b_n$ have been determined, use Lemma
\ref{lem3} to choose $a_{n+1} > b_n$ such that
$\theta^{-1}(\overline{\{\|g\|\in (b_{n},a_{n+1})\}}) \wedge
\overline{B(x_0,\ep_{n})} \neq 0$.  Then apply Lemma \ref{bump} to
obtain $b_{n+1} > a_{n+1}$ and $\tilde{g}_{n+1}\in \Lip_{\Sig'}(Y,F)$ with
$\om_{\tilde{g}_{n+1}} \leq 3\om_g$ so that
\[ \tilde{g}_{n+1}(y) = \begin{cases}
                  g(y) &\text{ if $\|g(y)\|\leq a_{n+1}$}\\
                  0 &\text{ if $\|g(y)\| \geq b_{n+1}$.}
                  \end{cases}
\]
For each $n$, let $G_n = \tilde{g}_{2n} - \tilde{g}_{2n-1}$.  Then
$G_n \in \Lip_{\Sig'}(Y,F)$ and $\om_{G_n} \leq 6\,\om_{g}$. Also $G_n(y) = 0$
if $\|g(y)\| \notin (a_{2n-1},b_{2n})$ and $G_n(y) = g(y)$ if
$\|g(y)\| \in [b_{2n-1},a_{2n}]$. In particular, $C(G_n)$ and
$C(G_m)$ are disjoint if $n \neq m$.  Thus the pointwise sum $G =
\sum G_{2m-1}$ is well defined and $\om_G \leq 12\om_g$ by Lemma \ref{lem2.1}. Hence $G \in \Lip_{\Sig'}(Y,F)$. For each $m$, let $V_m =
\overline{\{\|g\|\in (b_{2m-1},a_{2m})\}}$.  By the choice of
$a_{2m}$, one can find $x_m \in \theta^{-1}(V_m) \cap
\overline{B(x_0,\ep_{2m-1})}$.  Now $G = G_{2m-1} = g$ on
${V_{2m-1}}$ and $G = 0$ on ${V_{2m}}$ for all $m$. Hence $T^{-1}G =
T^{-1}g = \hat{u}$ on $\theta^{-1}({V_{2m-1}})$ and $T^{-1}G = 0$ on
$\theta^{-1}({V_{2m}})$ by  Proposition \ref{prop1.2}.  In
particular, $T^{-1}G(x_{2m-1}) = u \neq 0$ and $T^{-1}G(x_{2m}) = 0$
for all $m$.  Since the sequence $(x_m)$ converges to $x_0$ and
$T^{-1}G$ is continuous, we have reached a contradiction.  This
completes the proof of Proposition \ref{prop3}.
\end{proof}

\begin{lem}\label{lem6}
For any $x_0 \in X$, $\cap_{\ep > 0}\theta(\overline{B(x_0,\ep)})$ contains at most $1$ point.
\end{lem}

\begin{proof}
Suppose on the contrary that $y_1$ and $y_2$ are distinct points in
the intersection.  Let $f \in \Lip_{\Sig}(X,E)$ be such that $f(x_0) \neq 0$
and set $g = Tf$.
Choose $\delta
> 0$ so that $d'(y_1,y_2)
> 3\delta$ and that $g$ is bounded on $B(y_1,2\delta)$. Pick $\psi
\in \Lip_{\Sig'}(Y)$ so that $\psi = 1$ on $\overline{B(y_1,\delta)}$ and
$\psi = 0$ outside $B(y_1, 2\delta)$. Since $g$ is bounded on
$C(\psi)$, $\psi g\in \Lip_{\Sig'}(Y,F)$. As $\psi g= g$ on
$\overline{B(y_1,\delta)}$, $T^{-1}(\psi g) = f$ on
$\theta^{-1}(\overline{B(y_1,\delta)})$. Similarly, $T^{-1}(\psi g)
= 0$ on $\theta^{-1}(\overline{B(y_2,\delta)})$. For $i = 1,2$ and
any $\ep
> 0$, $\theta(\overline{B(x_0,\ep)}) \wedge \overline{B(y_i,\delta)}
\neq 0$ and hence $\overline{B(x_0,\ep)} \wedge
\theta^{-1}(\overline{B(y_i,\delta)}) \neq 0$.  By continuity of
$T^{-1}(\psi g)$ and $f$, we conclude that $T^{-1}(\psi g)(x_0) =
f(x_0)$ and $T^{-1}(\psi g)(x_0) = 0$; thus reaching a
contradiction.
\end{proof}

\begin{lem}\label{lem6.1}
Assume that $x_0\in X$ is an accumulation point.  Let $(U_n) =
(\overline{B(x_n,\ep_n)})$ be a sequence of pairwise disjoint sets,
where $(x_n)$ converges to $x_0$ and $(\ep_n)$ is a positive null
sequence. If $(y_n)$ is a sequence such that $y_n \in \inte
\theta(U_n)$ for each $n$, then $(y_n)$ has a Cauchy subsequence.
\end{lem}

\begin{proof}
If the lemma fails, by passing to a subsequence if necessary, we may
assume that there exists $\delta > 0$ such that $d'(y_m,y_n) >
3\delta$ whenever $m \neq n$.  There exists a $[0,1]$-valued $\psi \in
\Lip_{\Sig'}(Y)$ such that $\psi = 1$ on $\overline{B(y_n,\delta)}$ for all
odd $n$ and $\psi = 0$ on $\overline{B(y_n,\delta)}$ for all even
$n$. According to Proposition \ref{prop3}, there exists $f \in
\Lip_{\Sig}(X,E)$ such that $f(x_0) \neq 0$ and $Tf$ is bounded on $Y$. Then
$g = \psi Tf \in \Lip_{\Sig'}(Y,F)$.  Now $g = Tf$ on
$\overline{B(y_n,\delta)}$ for all odd $n$ and $g = 0$ on
$\overline{B(y_n,\delta)}$ for all even $n$. Hence $T^{-1}g = f$ on
$\theta^{-1}(\overline{B(y_n,\delta)})$ for all odd $n$ and $T^{-1}g
= 0$ on $\theta^{-1}(\overline{B(y_n,\delta)})$ for all even $n$.
Since $\theta(U_n) \wedge \overline{B(y_n,\delta)}\neq 0$ for all
$n$, $U_n \wedge \theta^{-1}(\overline{B(y_n,\delta)})\neq 0$ for
all $n$.  Therefore, we can find a sequence $(z_n)$ such that $z_n
\in U_n$ for all $n$, $T^{-1}g(z_n) = f(z_n)$ for odd $n$ and
$T^{-1}g(z_n) = 0$ for odd $n$.  This is impossible since $T^{-1}g$
and $f$ are continuous, $f(x_0) \neq 0$ and $(z_n)$ converges to
$x_0$.
\end{proof}

\begin{prop}\label{prop7}
For any $x_0 \in X$, $\cap_{\ep > 0}\theta(\overline{B(x_0,\ep)})$ contains exactly $1$ point.
\end{prop}

\begin{proof}
In view of Lemma \ref{lem6}, it suffices to prove that the
intersection in question is nonempty. If $x_0$ is an isolated point,
then $\{x_0\} \in \D(X)$ and $\theta(\{x_0\}) \subseteq \cap_{\ep >
0}\theta(\overline{B(x_0,\ep)})$.  So the proposition holds in this
case.

Assume that $x_0$ is an accumulation point.  Let $(U_n) =
(\overline{B(x_n,\ep_n)})$ be a sequence of pairwise disjoint sets,
where $(x_n)$ converges to $x_0$ and $(\ep_n)$ is a positive null
sequence. Pick a sequence $(y_n)$ with $y_n \in \inte \theta(U_n)$
for each $n$. By Lemma \ref{lem6.1}, $(y_n)$ has a Cauchy
subsequence. Relabeling, we may assume that $(y_n)$ is Cauchy and
hence converges to some $y_0 \in Y$.  For any $\ep > 0$, there
exists $n_0$ such that $U_n \subseteq \overline{B(x_0,\ep)}$ for all
$n \geq n_0$. Thus $y_n \in \theta(U_n) \subseteq
\theta(\overline{B(x_0,\ep)})$ for all $n \geq n_0$.  Since the
latter set is closed, $y_0 \in \theta(\overline{B(x_0,\ep)})$.
\end{proof}

Define $h(x_0)$ to be the unique point in $\cap_{\ep >
0}\theta(\overline{B(x_0,\ep)})$ for all $x_0 \in X$.  Similarly, we
may define $k : Y \to X$ by setting $k(y_0)$ to be the unique point
in $\cap_{\delta > 0}\theta^{-1}(\overline{B(y_0,\delta)})$.

\begin{prop}\label{prop8}
The map $h$ is a homeomorphism whose inverse is $k$.
\end{prop}

\begin{proof}
Suppose that $x_0 \in X$ and $h(x_0) = y_0$. For any $\ep, \delta >
0$, $\overline{B(y_0,\delta)}\wedge \theta(\overline{B(x_0,\ep)})
\neq 0$ and hence $\theta^{-1}(\overline{B(y_0,\delta)}) \wedge
\overline{B(x_0,\ep)} \neq 0$.  In particular, for any $\delta > 0$,
we can find $x_n \in \theta^{-1}(\overline{B(y_0,\delta)}) \wedge
\overline{B(x_0,1/n)}$ for each $n$.  Since
$\theta^{-1}(\overline{B(y_0,\delta)})$ is closed, $x_0 \in
\theta^{-1}(\overline{B(y_0,\delta)})$.  As $\delta > 0$ is
arbitrary, this shows that $k(y_0) = x_0$.  By symmetry, $h(k(y_0))
= y_0$ for all $y_0 \in Y$.

It remains to prove the continuity of $h$.  The continuity of $k$
follows by symmetry.  Let $x_0$ be a point in $X$.  Since $h$ is
trivially continuous at an isolated point, we may assume that $x_0$
is an accumulation point.  Let $(x_n)$ be a pairwise distinct
sequence converging to $x_0$.  Choose a positive null sequence
$(\ep_n)$ so that $(U_n) = (\overline{B(x_n,\ep_n)})$ is pairwise
disjoint. For each $n$, $h(x_n) \in \theta(U_n) = \overline{\inte
\theta(U_n)}$.  Hence there exists $y_n \in \inte\theta(U_n)$ so
that $d'(y_n,h(x_n)) < 1/n$. By Lemma \ref{lem6.1}, $(y_n)$ has a
subsequence converging to a point $y_0$ in $\cap_{\ep
> 0}\theta(\overline{B(x_0,\ep)})$. Consequently, $(h(x_n))$ has a
subsequence that converges to $y_0$. By Lemma \ref{lem6}, $y_0 =
h(x_0)$. The continuity of $h$ at $x_0$ follows.
\end{proof}

Observe that if $f = 0$ on an open set $U$ containing $x_0$, then $f
= 0$ on $\overline{B(x_0,\ep)}$ for some $\ep > 0$ and hence $Tf =
0$ on $\theta(\overline{B(x_0,\ep)})$.  In particular, $Tf(h(x_0)) =
0$.

\begin{prop}\label{prop10}
If $f \in \Lip_\Sig(X,E)$ and $f(x_0) = 0$, then $Tf(h(x_0)) = 0$.
\end{prop}

\begin{proof}
By the observation preceding the proposition, we only need to consider the case where  $x_0$ is an accumulation point of $C(f)$.
Suppose $z$ belongs to $C(f)$.  By Lemma \ref{bump}, there are functions $g_1,g_2: X\to E$ with $\om_{g_i} \leq 3 \om_f$, $i=1,2$, so that
\[ g_1(x) = \begin{cases}
            f(x) &\text{ if $\|f(x)\| \leq 2\|f(z)\|$}\\
            0 &\text { if $\|f(x)\| \geq 4\|f(z)\|$}
            \end{cases} \]
and
\[g_2(x) = \begin{cases}
            f(x) &\text{ if $\|f(x)\| \leq \frac{\|f(z)\|}{4}$}\\
            0 &\text { if $\|f(x)\| \geq \frac{\|f(z)\|}{2}$}.
            \end{cases}
\]
Set $g = g_1-g_2$.  Then $\om_g \leq 6\om_f$ and
\[ g(x) = \begin{cases}
            f(x) &\text{ if $\|f(x)\| \in [\frac{\|f(z)\|}{2},2\|f(z)\|]$}\\
            0 &\text { if $\|f(x)\| \notin [\frac{\|f(z)\|}{4},4\|f(z)\|]$}
            \end{cases}. \]
Let $(x_n)$ be a sequence in $C(f)$ converging to $x_0$ so that $\|f(x_{n+1})\|\leq 16\|f(x_n)\|$ for all $n$. For each $n$, let $f_n$ be the function $g$ described above with $z = x_{2n-1}$.
By Lemma \ref{lem2.1}, $\tilde{f} = \sum f_n$ belongs to $\Lip_\Sig(X,E)$.  For each $n$,
\[ A_n = \{x\in X: \frac{\|f(x_n)\|}{2} < \|f(x)\| < 2\|f(x_n)\|\} \]
is an open neighborhood of $x_n$.  Furthermore, $\tilde{f} = f$ on $\overline{A_n}$ if $n$ is odd and $\tilde{f} = 0$ on $\overline{A_n}$ if $n$ is even. By Proposition \ref{prop1.2}, $T\tilde{f} = Tf$
on $\theta(\overline{A_n})$ for odd $n$ and $T\tilde{f} = 0$ on
$\theta(\overline{A_n})$ for even $n$.
In particular, $T\tilde{f}(h(x_n)) = Tf(h(x_n))$ if $n$ is odd and $0$ if $n$ is even.  By continuity of $T\tilde{f}$, $Tf$ and $h$, we have
\begin{align*}
Tf(h(x_0)) &= \lim Tf(h(x_{2n-1})) = \lim T\tilde{f}(h(x_{2n-1})) \\&=
T\tilde{f}(h(x_0)) = \lim T\tilde{f}(h(x_{2n})) = 0.
\end{align*}
\end{proof}

The following is the main result of this section.  It includes as a
special case the result of \cite{AD} characterizing biseparating
maps between spaces of bounded Lipschitz functions.

\begin{thm}\label{thm11}
Let $X, Y$ be complete metric spaces and $E, F$ be Banach spaces.
Suppose that $\Lip_\Sig(X,E)$ and $\Lip_{\Sig'}(Y,F)$ are
generalized Lipschitz spaces that are Lipschitz normal. If $T :
\Lip_{\Sig}(X,E) \to \Lip_{\Sig'}(Y,F)$ is a linear biseparating
map, then there exist a homeomorphism $h : X \to Y$ and, for each $y
\in Y$, a vector space isomorphism $S_y : E \to F$ such that
\begin{equation}\label{eq2}
Tf(y) = S_y(f(h^{-1}(y))) \quad \text{for all $y \in Y$}.
\end{equation}
\end{thm}

\begin{proof}
Let $h: X\to Y$ be defined as above.  Then $h$ is a homeomorphism by
Proposition \ref{prop8}.  Define $S_y : E\to F$ by $S_yu
= T\hat{u}(y)$ for all $y \in Y$.  If $f \in \Lip_{\Sig}(X,E)$ and $y \in Y$,
then $(f - \hat{u})(h^{-1}(y)) = 0$, where $u = f(h^{-1}(y))$.
Therefore, $Tf(y) = T\hat{u}(y)$ by Proposition \ref{prop10}.  Thus
(\ref{eq2}) holds.  The linearity of $S_y$ follows from that of $T$.
If $v \in F$, there exists $f$ such that $Tf = \hat{v}$.  Hence, for any $y$,
taking $u = f(h^{-1}(y))$, we find that $S_yu = v$. This shows
that each $S_y$ is onto.  Finally, if $S_yu = 0$, then
$T\hat{u}(y) = 0$.  Applying Proposition \ref{prop10} to $T^{-1}$
and $h^{-1}$, we find that $u = T^{-1}T\hat{u}(h^{-1}(y)) = 0$. Thus
$S_y$ is one-to-one.
\end{proof}

\section{Continuity}

In this section, let $T: \Lip_\Sig(X,E) \to \Lip_{\Sig'}(Y,F)$ be biseparating map.  Thus $T$ has the form given in (\ref{eq2}) of
Theorem \ref{thm11}, where $h$ is a homeomorphism and $S_y$ is a vector space isomorphism for all $y \in Y$.  We investigate the continuity properties of the
family $(S_y)$ and of the operator $T$ with respect to suitable
topologies.   We also consider the metric properties of the mapping
$h$.

\begin{prop}\label{prop12}
If $y_0$ is an accumulation point in $Y$, then $S_{y_0}$ is a bounded linear operator. Furthermore, if
\begin{equation}\label{eq}
\sup_{\sigma\in \Sig'}\sigma(\ep) < \infty \text{ for all $\ep \geq 0$},
\end{equation}
then $S_y$ is bounded at all $y \in Y$ except for finitely many isolated points of $Y$.
\end{prop}

\begin{proof}
Assume on the contrary that $S_{y_0}$ is unbounded for some accumulation point
$y_0$ of $Y$. Set $x_0 = h^{-1}(y_0)$. Then $x_0$ is an accumulation point of $X$.
Choose $\vp_1 \in \Lip_\Sig(X)$ with values in $[0,1]$ so that $\vp_1(x_0) = 1$ and $\vp_1 = 0$ outside $B(x_0,1)$.  There are $K_1 < \infty$ and $\sigma_1\in \Sig$ so that $\om_{\vp_1}\leq K_1\sigma_1$.  Pick a norm-$1$ vector $u_1$ in
$E$ and let $f_1(x) = \vp_1(x)u_1$.   Since $Tf_1\circ h$ is continuous on $X$ and $Tf_1(h(x_0))
= S_{y_0}u_1$ by (\ref{eq2}), there exists $r_1 \in (0,1)$ such that
$\|Tf_1(h(x)) - S_{y_0}u_1\| \leq 1$ for all $x \in B(x_0,r_1)$.
Choose $x_1 \in B(x_0,r_1)\bs\{x_0\}$. In general, after $x_{n-1}$
and $(u_m)^{n-1}_{m=1}$ have been determined, let $\vp_n \in \Lip_\Sig(X)$ be a $[0,1]$-valued function such that $\vp_n(x_0)=1$ and $\vp_n =0$ outside $B(x_0, d(x_{n-1},x_0))$.
There are $K_n < \infty$ and $\sigma_n\in \Sig$ so that $\om_{\vp_n}\leq K_n\sigma_n$.  Of course, $\om_{\vp_n} \leq 1$ as well. Set $f_n(x) = \vp_n(x)u_n$ for a vector $u_n \in E$ such that $\|u_n\| = (n^2(K_n+1))^{-1}$ and $\|S_{y_0}u_n\| \geq
\sum^{n-1}_{m=1}\|S_{y_0}u_m\| + 2n$.
There exists $r_n$, $0< r_n < d(x_{n-1},x_0)\wedge n^{-1}$ such that
$\|Tf_n(h(x)) - S_{y_0}u_n\| \leq 1$ for all $x \in B(x_0,r_n)$.
Choose $x_n \in B(x_0,r_n)\bs\{x_0\}$. This completes the inductive construction.

Since $\|f_n(x)\| \leq \|u_n\| \leq 1/n^2$ for all $n$ and all $x \in X$, $f = \sum f_n$ exists.  Furthermore, for all $n$,
\[ \om_{f_n} \leq \frac{\om_{\vp_n}}{n^2(K_n+1)} \leq \frac{1}{n^2(K_n+1)}(K_n\sigma_n\wedge 1)\leq  \frac{1}{n^2}(\sigma_n \wedge 1).\]
By condition (MS2) in the definition of modulus sets, we find $K < \infty$ and $\sigma\in \Sig$ so that $\sum n^{-2}(\sigma_n\wedge 1) \leq K\sigma$.  Hence $\om_f \leq K \sigma$ and thus $f \in \Lip_\Sig(X,E)$.
If $m > n$, we have $d(x_n,x_0) \geq d(x_{m-1},x_0)$ and hence $f_m(x_n) = 0$.
For all $n$,
\[
Tf(h(x_n)) = S_{h(x_n)}(f(x_n)) = \sum^n_{m=1}S_{h(x_n)}(f_m(x_n)) = \sum^n_{m=1}Tf_m(h(x_n)).\]
Therefore,
\begin{align*}
\|Tf(h(x_n))\| &\geq \|Tf_n(h(x_n))\| - \sum^{n-1}_{m=1}\|Tf_m(h(x_n))\|\\
&\geq \|S_{y_0}u_n\| - 1 - \sum^{n-1}_{m=1}(\|S_{y_0}u_m\| + 1)\\
&\geq \sum^{n-1}_{m=1}\|S_{y_0}u_m\| + 2n - 1 - \sum^{n-1}_{m=1}(\|S_{y_0}u_m\| + 1) = n.
\end{align*}
However,  $(x_n)$ converges to $x_0$ and thus
$(Tf(h(x_n)))$ converges.  We have reached a contradiction.

Now suppose that (\ref{eq}) holds and that $(y_n)$ is an infinite
sequence so that $S_{y_n}$ is unbounded for all $n$.  By the above,
each $y_n$ is isolated in $Y$, and consequently each $x_n =
h^{-1}(y_n)$ is isolated in $X$. Since $\Lip_\Sig(X)$ is Lipschitz
normal, for each $n$, the function $\vp_n$ defined by $\vp_n(x_n) =
1$ and $\vp_n(x) = 0$ otherwise belongs to $\Lip_\Sig(X)$.  Let $K_n
< \infty$ and $\sigma_n \in \Sig$ be such that $\om_{\vp_n} \leq
K_n\sigma_n$. For each $n \geq 2$, let $b_n = d'(y_n,y_1)$ and
choose $u_n \in E$ with $\|u_n\| = (n^2(K_n+1))^{-1}$ and
$\|S_{y_n}u_n\| > n \sup_{\sigma\in \Sig'}\sigma(b_n)$.  Let $f(x)$
be the pointwise sum $\sum^\infty_{n=2}\vp_n(x)u_n$.  Then
\[ \om_f \leq \sum^\infty_{n=1}\frac{\om_{\vp_n}}{n^2(K_n+1)} \leq \sum^\infty_{n=2}\frac{1}{n^2}(\sigma_n\wedge 1) \leq
K\sigma_0\] for some $K< \infty$ and $\sigma_0 \in \Sig$.  Hence $f
\in \Lip_\Sig(X,E)$. It follows that $Tf \in \Lip_{\Sig'}(Y,E)$ and
so there are $K' < \infty$ and $\sigma' \in \Sig'$ so that $\om_{Tf}
\leq K'\sigma'$. But $Tf(y_n) = S_{y_n}u_n$ for
all $n>1$ and $Tf(y_1) = 0$.  Therefore,
\[ K'\sigma'(b_n) \geq \om_{Tf}(b_n) \geq \|Tf(y_n) - Tf(y_1)\| = \|S_{y_n}u_n\| > n \sup_{\sigma\in
\Sig'}{\sigma(b_n)}\]
for all $n \geq 2$, which is clearly impossible.
\end{proof}

\noindent{\bf Remark}. If $\Lip_{\Sig'}(Y,F)$ consists of bounded
functions only, then we may replace $\Sig'$ with $\{\sigma\wedge 1:
\sigma \in \Sig'\}$.  In this case, (\ref{eq}) is fulfilled and
hence $S_y$ is bounded except for finitely many isolated points of
$Y$. The special case of the result for spaces of bounded Lipschitz
functions was obtained in \cite{AD}. Condition (\ref{eq}) also holds
for the spaces $\Lip(Y,F)$ and $\lip(Y,F)$.
\\

Let $Y_1$ be the set of all $y\in Y$ at which $S_y$ is bounded and
$X_1$ be the set of all $x \in X$ at which $S^{-1}_{h(x)}$ is
bounded.  The next result is a simple application of the Uniform
Boundedness Principle.

\begin{cor}\label{cor13}
Let $r > 0$ be a real number such that $\sigma'(r) < \infty$ for all $\sigma'\in \Sigma'$.  Then $\{S_y: y \in Y_1, d(y,y_0) \leq r\}$
is uniformly bounded for any $y_0 \in Y$.
Similarly, $\{S_y:y\in Y_1\}$ is uniformly bounded if $\sup_{r> 0}\sigma'(r) < \infty$ for all $\sigma'\in \Sigma'$.
\end{cor}

\begin{proof}
For any $u \in E$, there are $K< \infty$ and $\sigma' \in \Sig'$ such that $\om_{T\hat{u}} \leq K \sigma'$.  If $d(y,y_0) \leq r$, then
\[ \|S_yu-S_{y_0}u\| = \|T\hat{u}(y) - T\hat{u}(y_0)\| \leq K\sigma'(r) < \infty.\]
Hence $\{S_yu: y \in Y_1, d(y,y_0) \leq r\}$ is bounded for all $u \in E$. Thus $\{S_y: y \in Y_1, d(y,y_0) \leq r\}$ is uniformly bounded by the Uniform Boundedness Principle. The second statement is proved similarly.
\end{proof}

In general, we still have ``local uniform boundedness".

\begin{prop}\label{prop13.2}
For all $y_0 \in Y_1$, there is a neighborhood $V$ of $y_0$ in $Y_1$ so that $\{S_y: y\in V\}$ is uniformly bounded.
\end{prop}

\begin{proof}
If the proposition fails, there are sequences $(y_n)$ in $Y_1$ converging to $y_0$ and $(u_n)$ in $E$ with $\|u_n\| = 1/2^n$ so that $\|S_{y_n}u_n\| \to \infty$.
Note that $(S_{y_n}u_m)_m$ converges to $0$ for each $n$ and $\lim_nS_{y_n}u_m = \lim_nT\hat{u}_m(y_n) = T\hat{u_m}(y_0)$ for each $m$.  Thus, by passing to subsequences, we may assume that, for each $n$,
$\|S_{y_n}u_n\| \geq n + \sum_{m\neq n}\|S_{y_n}u_m\|$.
Let $w = \sum u_n$.  Then
\[ \|T\hat{w}(y_n)\| = \|S_{y_n}w\| \geq \|S_{y_n}u_n\| - \sum_{m\neq n}\|S_{y_n}u_m\| \geq n\]
for all $n$.  However, $T\hat{w}$ is continuous and so $(T\hat{w}(y_n))$ converges, contrary to the above.
\end{proof}

Next we consider the continuity of $T$.  First we look at a
diagonalization lemma.

\begin{lem}\label{lem4.20}
Let $(g_n)$ be a sequence of functions from $Y$ into $F$. Suppose
that there are a positive sequence $(c_n)$, sequences $(y^n_1)$,
$(y^n_2)$ in $Y$ and $C < \infty$ so that
\begin{align*}
\sup_n\frac{\|g_n(y^n_1) - g_n(y^n_2)\|}{c_n} &= \infty,\\
\sup_n\frac{\|g_m(y^n_1)-g_m(y^n_2)\|}{c_n} &= L_m < \infty \text{
for all $m$}, \\
\intertext{and}
\sup_{m,n}\|g_m(y^n_1) - g_m(y^n_2)\| &= C < \infty.
\end{align*}
Then  there exists a nonnegative summable sequence $(\ep_n)$ so that if
the pointwise sum $g = \sum \ep_ng_n$ converges on the set
$\{y^n_1,y^n_2: n \in \N\}$, then $\sup_n\|g(y^n_1) -
g(y^n_2)\|/{c_n} = \infty$.
\end{lem}

\begin{proof}
Let $K_n = \|g_n(y^n_1) - g_n(y^n_2)\|/{c_n}$.  Choose $n_1 < n_2 < \cdots$ and a summable
sequence $(\ep_{n_k})$ so that $\ep_{n_k}K_{n_k} \geq
\max\{3\sum^{k-1}_{m=1}\ep_{n_m}L_{n_m},k\}$ and $3C\sum^\infty_{m=k+1}\ep_{n_m}
\leq \ep_{n_k}K_{n_k}c_{n_k}$ for all $k$. Define $\ep_n = 0$ if $n \neq n_k$ for any $k$. If $g = \sum \ep_ng_n$
converges pointwise on $\{y^n_1,y^n_2: n \in \N\}$,
\begin{align*}
\|g(y^{n_k}_1) &- g(y^{n_k}_2)\|\\
&\geq \ep_{n_k}{\|g_{n_k}(y^{n_k}_1) - g_{n_k}(y^{n_k}_2)\|} -
\sum^{k-1}_{m=1}\ep_{n_m}L_{n_m}c_{n_k} - C\sum^\infty_{m=k+1}\ep_{n_m}.
\end{align*}
 Thus
\[ \|g(y^{n_k}_1) - g(y^{n_k}_2)\| \geq \ep_{n_k}K_{n_k}c_{n_k} - \frac{\ep_{n_k}K_{n_k}c_{n_k}}{3}
- \frac{\ep_{n_k}K_{n_k}c_{n_k}}{3} \geq \frac{kc_{n_k}}{3}\]
for all $k$.
\end{proof}

The next proposition is a form of continuity of the operator $T$. For a function $f \in \Lip_\Sig(X,E)$ and a subset $U$ of $X$, let $\|f\|_U = \sup\{\|f(x)\| : x\in U\}$.

\begin{prop}\label{prop4.23}
Suppose that $U$ is a subset of $X$ such that
$V = h(U) \subseteq Y_1$ and that $M =
\sup\{\|S_y\|: y \in V\} < \infty$. Then there exists $K < \infty$
so that for any $f\in \Lip_\Sig(X,E)$ with $\om_f \leq \sigma$ for
some $\sigma \in {\Sig}$ and $\|f\|_U \leq 1$, we have $\om^V_{Tf} \leq K\sup_{\sigma'\in
\Sig'}\sigma'$, where
\[\om^V_{Tf}(t) =
\sup\{\|Tf(y_1)-Tf(y_2)\|: y_1,y_2 \in V, d'(y_1,y_2)\leq t\}.\]
\end{prop}

\begin{proof}
Otherwise, for all $n$, we find $\overline{f}_n \in \Lip_\Sig(X,E)$,
$\sigma_n \in {\Sig}$ and $t_n> 0$ such that $\om_{\overline{f}_n}
\leq \sigma_n$, $\|\ol{f}_n\|_U\leq 1$ and
$\sup_n[\om^V_{T\overline{f}_n}(t_n)/\sup_{\sigma'\in\Sig'}\sigma'(t_n)]
= \infty$. Let  $c_n = \sup_{\sigma'\in\Sig'}\sigma'(t_n)$. There
are $(y^n_1,y^n_2) \in V\times V$ such that $d'(y^n_1,y^n_2) \leq t_n$
and that
$\sup_n\|T\overline{f}_n(y^n_1)-T\overline{f}_n(y^n_2)\|/c_n =
\infty$. Let $x^n_i = h^{-1}(y^n_i) \in U$, $i = 1,2$.  Consider the retraction $R: E \to B_E$ as in Lemma
\ref{lem1}. Set $f_n = R\circ \overline{f}_n$.  Then $\om_{f_n} \leq
2\om_{\overline{f}_n} \leq 2\sigma_n$ and $\om_{f_n}\leq 2$. Let
$g_n = Tf_n$ for all $n$.  Then, for all $m,n$,
\[ \|g_m(y^n_1) - g_m(y^n_2)\| = \|S_{y^n_1}f_m(x^n_1) -
S_{y^n_2}f_m(x^n_2)\|.\] %
Since $\overline{f}_n(x^n_1), \overline{f}_n(x^n_2)
\in B_E$, $S_{y^n_i}f_n(x^n_i) = S_{y^n_i}\overline{f}_n(x^n_i) =
T\overline{f}_n(x^n_i)$ for $n \in \N$, $i = 1,2$.  Thus
\[\sup_ n\frac{\|g_n(y^n_1) - g_n(y^n_2)\|}{c_n} = \infty.\]
Since $g_m \in \Lip_{\Sig'}(Y,F)$,
\[ \sup_ n\frac{\|g_m(y^n_1) - g_m(y^n_2)\|}{c_n} < \infty \text{
for all $m$}.\]%
Moreover, for all $m$ and $n$,
\[ \|g_m(y^n_1)- g_m(y^n_2)\| \leq M(\|f_m(x^n_1)\| +
\|f_m(x^n_2)\|) \leq 2M.
\]
Therefore, Lemma \ref{lem4.20} applies to $(g_n)$ and we obtain a
summable sequence $(\ep_n)$ as in the lemma.  Since $\|f_n(x)\| \leq
1$ for all $n$ and $x$, $f = \sum \ep_n f_n$ converges pointwise on
$X$.  Also,
\[\om_f \leq \sum\ep_n\om_{f_n} \leq
\sum\ep_n(2\sigma_n\wedge 2) = 2\sum \ep_n(\sigma_n\wedge 1).\]
By (MS2), $f \in \Lip_{\Sig}(X,E)$ and hence $Tf \in
\Lip_\Sig(X,E)$. Since $S_y$ is bounded for all $y \in V$, for all
such $y$,
\[ Tf(y) = S_{y}f(h^{-1}(y)) = \sum\ep_nS_{y}f_n(h^{-1}(y))
= \sum\ep_ng_n(y).\]
But by the conclusion of Lemma \ref{lem4.20}, $\sup_n\|Tf(y^n_1) -
Tf(y^n_2)\|/c_n = \infty$, contradicting the fact that $Tf \in
\Lip_{\Sig'}(Y,F)$.
\end{proof}

If $0<\alpha \leq 1$, then $\Lip_\alpha(X,E) =
\Lip_{\{\sigma_\alpha\}}(X,E)$, where $\sigma_\alpha(t) = t^\alpha$. Similarly, $\lip_\alpha(X,E) = \Lip_\Sigma(X,E)$, where $\Sigma$ consists of all modulus functions $\sigma$ such that $\sigma(t) \leq t^\alpha$ and $\lim_{t\to 0}\sigma(t)/t^\alpha = 0$.
Define the
$\Lip_\alpha$ constant of $f\in \Lip_\alpha(X,E)$ to be $L_\alpha(f) =
\sup\{\om_f(t)/t^\alpha: t > 0\}$.  For spaces of Lipschitz ($\Lip$)
functions, the next corollary was obtained in \cite{AD}.

\begin{cor}\label{cor4.24}
Let $T$ be a biseparating map from $\Lip(X,E)$ onto $\Lip(Y,F)$,
respectively, from $\lip_\alpha(X,E)$ onto $\lip_\alpha(Y,F)$.
Suppose that $U$ is a bounded subset of $X$ such that $V = h(U)$ is
a bounded subset of $Y_1$. Then there exists $K < \infty$ so that
$L_1(Tf_{|V}) \leq K(L_1(f)\vee \|f\|_U)$ for all $f \in \Lip(X,E)$, respectively,
$L_\alpha(Tf_{|V}) \leq K(L_\alpha(f)\vee \|f\|_U)$ for all $f \in
\lip_\alpha(X,E)$.
\end{cor}

\begin{proof}
Observe that Proposition \ref{prop4.23} applies since $\{S_y: y\in
V\}$ is uniformly bounded by Corollary \ref{cor13}.
\end{proof}

Recall that if $\Sig$ is a modulus set and $\Sig_b = \{\sigma\wedge 1: \sigma \in \Sig\}$, then $\Lip_{\Sig_b}(X,E)$ is precisely the space of all bounded functions in $\Lip_{\Sig}(X,E)$.

\begin{cor}\label{cor4.22}
Let $T: \Lip_{\Sig_b}(X,E) \to \Lip_{\Sig'_b}(Y,F)$ be a biseparating map. There exists $K < \infty$ so that for any $f \in \Lip_{\Sig_b}(X,E)$ with $\om_f \leq \sigma\wedge 1$ for some $\sigma\in \Sig$ and $\|f\|_X \leq 1$, we have $\om^{Y_1}_{Tf} \leq K\sup_{\sigma'\in \Sig}(\sigma'\wedge 1)$.
\end{cor}

\begin{proof}
By Corollary \ref{cor13}, $\sup\{\|S_y\|: y \in Y_1\} < \infty$ in this case. The corollary follows immediately from Proposition \ref{prop4.23}.
\end{proof}

Denote by $\U_b(X,E)$ the space of bounded uniformly continuous functions from $X$ into $E$.

\begin{cor}\label{cor4.23}
Let $T:\U_b(X,E) \to \U_b(Y,F)$ be a biseparating map.  Then there exists a finite set $I$ of isolated points of $Y$ and $K < \infty$ such that
$\sup_{y \in Y\bs I}\|Tf(y)\| \leq K \sup_{x\in X}\|f(x)\|$ for all $f \in \U_b(X,E)$.
\end{cor}

\begin{proof}
Let $\Sig = \Sig'$ be the set of all modulus functions.  Then $\U_b(X,E) = \Lip_{\Sig_b}(X,E)$ and $\U_b(Y,F) = \Lip_{\Sig'_b}(Y,F)$. Recall that $Y_1$ is the set of all $y \in Y$ at which $S_y$ is bounded. In this case, $I = Y\bs Y_1$ consists of finitely many isolated points of $Y$ by Proposition \ref{prop12}.
Fix  $x_0 \in X$ so that $y_0 = h(x_0) \in Y_1$. Suppose that $f \in \U_b(X,E)$ with $\sup_{x\in X}\|f(x)\|\leq 1$.  By Corollary \ref{cor4.22}, there exists $K < \infty$ so that for all $y \in Y_1$,
\begin{align*}
\|Tf(y)\| & \leq \|Tf(y)-Tf(y_0)\| + \|Tf(y_0)\|\\
&\leq \om^{Y_1}_{Tf}(d'(y,y_0)) + \|S_{y_0}f(x_0)\|\\
& \leq K +\|S_{y_0}\|.
\end{align*}
\end{proof}

We also obtain local continuity of $T$ with respect to the
topology of uniform convergence.

\begin{prop}\label{prop4.25}
Let $T: \Lip_\Sig(X,E) \to \Lip_{\Sig'}(Y,F)$ be a biseparting map. If $x_0 \in X$ and
$y_0 = h(x_0) \in Y_1$, then there exists a neighborhood $V$ of
$y_0$ in $Y_1$ and $K< \infty$ so that, setting $U = h^{-1}(V)$, we
have $\sup_{y\in V}\|Tf(y)\| \leq K$ for all $f \in \Lip_\Sig(X,E)$
such that
$\sup_{x \in U}\|f(x)\| \leq 1$.
\end{prop}

\begin{proof}
By Proposition \ref{prop13.2}, there is a neighborhood $V$ of $y_0$
in $Y_1$ such that $\sup\{\|S_y\|: y \in V\} = K < \infty$.  Set $U =
h^{-1}(V)$.
If $f \in \Lip_\Sig(X,E)$ and
$\sup_{x \in U}\|f(x)\| \leq 1$, then for all $y \in V$,
$\|Tf(y)\| = \|S_yf(h^{-1}(y))\| \leq K$.
\end{proof}

We can now deduce the metric properties of the map $h$. For any
$x_1, x_2 \in X$, define
\[ s(x_1,x_2) = \sup\{\|f(x_2)\|: f(x_1) = 0, \om_f \leq \sigma \text{ for some } \sigma \in \Sig\}.\]

\begin{prop}\label{prop4.20}
Let $U$ be a subset of $X_1$ so that $V= h(U)$ is a subset of $Y_1$.
Assume $\sup\{\|S_{y}\|,\|S^{-1}_y\|: y \in V\} = M < \infty$ and
that $\sup\{s(x_1,x_2): x_1,x_2 \in U\}=C <\infty$. There exists $K
< \infty$ such that $s(x_1,x_2) \leq K
\sup_{\sigma'\in\Sig'}\sigma'(d'(h(x_1),h(x_2)))$ for all $x_1, x_2
\in U$.
\end{prop}

\begin{proof}
Appeal to Proposition \ref{prop4.23} to find  $K <
\infty$ so that for any $f\in \Lip_\Sig(X,E)$ with $\om_f \leq
\sigma$ for some $\sigma \in {\Sig}$ and $\|f\|_U \leq 1$, we have $\om^V_{Tf} \leq
K\sup_{\sigma'\in \Sig'}\sigma'$.  If $a < s(x_1,x_2)$, $x_1, x_2
\in U$, choose $f \in \Lip_\Sig(X,E)$ and $\sigma \in \Sig$ so that
$f(x_1) = 0$, $\|f(x_2)\| > a$ and $\om_f \leq \sigma$.  Now $x\in U$ implies that $\|f(x)\| \leq s(x_1,x) \leq C$.  Let $\ol{f} = f/(C\vee 1)$. Then $\om_{\ol{f}} \leq \sigma$ and $\|\ol{f}\|_U \leq 1$.  Thus
\begin{align*}
\frac{a}{C\vee 1} &< \|\ol{f}(x_2)\| \leq M{\|S_{h(x_2)}\ol{f}(x_2)\|} =
M{\|T\ol{f}(h(x_1)) - T\ol{f}(h(x_2))\|}\\
&\leq M{\om^V_{T\ol{f}}(d'(h(x_1),h(x_2)))} \leq M{K
\sup_{\sigma'\in\Sig'}\sigma'(d'(h(x_1),h(x_2)))}.
\end{align*}
This proves the proposition.
\end{proof}

\begin{cor}\label{cor4.21} \cite{AD}
Let $T : \Lip(X,E) \to \Lip(Y,F)$ be a biseparating map and $h:X \to
Y$ be the homeomorphism associated to $T$. For any bounded set $U
\subseteq X$ such that $V = h(U)$ is bounded in $Y$, $h$ is
Lipschitz on $U$ and $h^{-1}$ is Lipschitz on $V$.
\end{cor}

\begin{proof}
By Proposition \ref{prop12}, the set $Y_1$ consisting of all $y\in
Y$ where $S_y$ is bounded contains all of $Y$ except for finitely
many isolated points of $Y$.  A similar statement holds for $X_1$.
Thus it suffices to show that $h$ is Lipschitz on $U_1 = U \cap X_1
\cap h^{-1}(Y_1)$ and that $h^{-1}$ is Lipschitz on $V_1 = h(U_1)$.
By Corollary \ref{cor13}, $\sup\{\|S_y\|, \|S^{-1}_y\|: y \in V_1\}
< \infty$. For all $x_1,x_2 \in X$, it is clear that the quantity
$s(x_1,x_2)$ defined above has the value $d(x_1,x_2)$.  In
particular, $\sup\{s(x_1,x_2): x_1,x_2 \in V_1\}< \infty$ since
$V_1$ is bounded. By Proposition \ref{prop4.20}, there exists $K
<\infty$ such that $d(x_1,x_2) \leq Kd'(h(x_1),h(x_2))$ for all
$x_1,x_2 \in U_1$.  This shows that $h^{-1}$ is Lipschitz on $V_1$.
The proof for $h$ is similar.
\end{proof}

\noindent{\bf Remark}. If $0 < \alpha < 1$, then for a bounded subset $U$ of the space
$\lip_\alpha(X,E)$, there is a positive constant $K_\alpha$ so that
$s(x_1,x_2) \geq K_\alpha d^\alpha(x_1,x_2)$ for all $x_1, x_2 \in U$.  Hence Corollary
\ref{cor4.21} also applies to biseparating maps between
$\lip_\alpha$ spaces, $0 < \alpha < 1$.

\begin{prop}\label{prop4.26}
Let $T:\Lip_{\Sig}(X,E) \to \Lip_{\Sig'}(Y,F)$ be a biseparating map.  Suppose that $\Lip(Y,F) \subseteq \Lip_{\Sig'}(Y,F)$.  If $U$ is a subset of $X$ so that $V = h(U) \subseteq Y_1$ and $\sup_{y\in V}\|S_y\| < \infty$, then $h$ is uniformly continuous on $U$.
\end{prop}

\begin{proof}
If the proposition fails, we can find a sequence $((x^n_1,x^n_2))_n$ in $U\times U$ and an $\ep > 0$ so that $\lim d(x^n_1,x^n_2)=0$ and $d'(y^n_1,y^n_2) \geq \ep$ for all $n$, where $y^n_i = h(x^n_i)$. If $(y^n_1)$ has a convergent subsequence, then by continuity of $h^{-1}$ and the fact that $\lim d(x^n_1,x^n_2)=0$, there is a subsequence $I$ of $\N$ so that $(x^n_1)_{n\in I}$ and $(x^n_2)_{n\in I}$ both converge to the same $x_0$.  Then $(y^n_1)_{n\in I}$ and $(y^n_2)_{n\in I}$ must both converge to $h(x_0)$, contrary to their choice.
Thus $(y^n_1)$, and, by symmetry, $(y^n_2)$ cannot have convergent subsequences.  Without loss generality, there exists $\delta$, $0 < 4\delta < \ep$ so that $d(y^n_1,y^m_1), d(y^n_2,y^m_2) > 2\delta$ for all $m \neq n$.  Then the sets $B(y^n_1,\delta), n \in \N$, are pairwise disjoint and each can contain at most one $y^m_2$, in which case $m \neq n$.
Hence we can choose a subsequence $J$ of $\N$ so that $y^n_2 \notin B(y^m_1,\delta)$ if $n,m \in J$.   Pick a normalized vector $v \in F$ and let
$g(y) = v\cdot\sup_{n\in J}(\delta- d(y,y^n_1))^+$.
Then $g \in \Lip(Y,F) \subseteq\Lip_{\Sig'}(Y,F)$ and hence $f = T^{-1}g \in \Lip_\Sig(X,E)$.  In particular, $f$ is uniformly continuous.
However, for all $n\in J$,
\[
\|f(x^n_1) - f(x^n_2)\| = \|S_{y^n_1}^{-1}g(y^n_1) - S_{y^n_2}^{-1}g(y^n_2)\| = \delta\|S_{y^n_2}^{-1}v\|\geq \delta\|S_{y^n_2}\|^{-1}.\]
Since $\sup\|S_{y^n_2}\| < \infty$ and $d(x^n_1,x^n_2)\to 0$, $f$ cannot be uniformly continuous.
\end{proof}

\end{document}